\documentclass[12pt]{article}
\usepackage{amsmath, amsfonts, amsthm, amssymb, color, hyperref, extarrows, enumitem, nicefrac,blindtext, mathtools, cite, anyfontsize, comment}

\newtheorem{theorem}{Theorem}[section]

\newtheorem{cor}[theorem]{Corollary}
\newtheorem{lem}[theorem]{Lemma}

\numberwithin{equation}{section}

\newtheorem{remark}[theorem]{Remark}

\newtheorem{example}{Example}

\usepackage[hyperpageref]{backref}

\usepackage{cite}

\hypersetup{
	colorlinks   = true,
	citecolor    = magenta}
	
\begin{document}
\title{\vspace{-1cm} \bf Unique continuation   for $\bar\partial u = Vu$ at infinity  }
\author{Yifei Pan \ \ and \ \  Yuan Zhang}
\date{}

\maketitle

\begin{abstract}
Motivated by  Landis's conjecture on unique continuation at infinity for the Laplacian, we study the corresponding  property for the Cauchy-Riemann operator. We prove that every weak solution of $\bar\partial u=Vu$ on a neighborhood of infinity, with $V\in L^\infty$, vanishes identically if it decays exponentially at a rate greater than $ 2\|V\|_{L^\infty}$. This conclusion is sharp both in the constant \(2\|V\|_{L^\infty}\) and in the order of exponential decay required. More generally, we establish unique continuation at infinity for a broad class of radially decaying bounded potentials, with optimal decay rates determined by the decay of the potential.   We also obtain related unique continuation results  for $L^2$ potentials and for compactly supported potentials under weaker   assumptions at infinity.

\end{abstract}

\renewcommand{\thefootnote}{\fnsymbol{footnote}}
\footnotetext{\hspace*{-7mm}
\begin{tabular}{@{}r@{}p{16.5cm}@{}}
& 2020 Mathematics Subject Classification. Primary 32W05; Secondary 35B60. \\
& Key words and phrases.
 Landis conjecture, Cauchy-Riemann operator, unique continuation, infinity. 
\end{tabular}}

\section{Introduction}

Let $u$ solve 
\begin{equation}\label{la}
     -\Delta u = V u
\end{equation} on $\mathbb R^d, d\ge 2$, where    $V\in L^\infty(\mathbb R^d)$. In 1988, Landis \cite{La} conjectured that if   $|u(x)|\le   e^{-|x|^{1+\epsilon}}$   for some $\epsilon>0$  whenever $|x|\gg1$, then $u\equiv 0$. This may be viewed as a unique continuation property at infinity  for solutions of \eqref{la} satisfying sufficiently rapid decay. The conjecture fails in general for  complex-valued potentials. Indeed,      Meshkov \cite{Me} constructed a nontrivial solution $u_0$ of \eqref{la} in $\mathbb R^2$ for some complex-valued potential $V\in L^\infty(\mathbb R^2)$ such that   $ |u_0(x)|\le  e^{-k|x|^{\frac{4}{3}}}$ for some $k>0$    whenever $|x|\gg1$. By contrast, the  case of real-valued potentials has attracted  considerable attention, particularly following the fundamental works of Bourgain--Kenig \cite{BK} and Kenig  \cite{Ke}.  More recently, Logunov, Malinnikova, Nadirashvili and  Nazarov \cite{LMNN} established the   Landis conjecture for real-valued potentials  in dimension two. In a recent preprint,  Frank and Ivanisvili \cite{FI} constructed    a nontrivial   solution of \eqref{la} with a bounded real-valued potential that decays like $e^{-k|x|^\frac{4}{3}}$ in dimension three and higher, thereby providing a counterexample to the Landis conjecture in these dimensions.

The purpose of this  paper is to investigate analogous unique continuation properties at infinity for the Cauchy-Riemann operator  $\bar\partial$. More precisely,   we consider complex-valued  functions  on a domain $\Omega\subset \mathbb C^n$ containing a neighborhood of infinity and satisfying
  \begin{equation}\label{eqn}
        \bar\partial u = Vu 
    \end{equation}
 on $\Omega$ in the sense of distributions,  where $V$ is a measurable $(0,1)$-form on $\Omega$.   The main objective is to determine whether sufficiently rapid decay of a solution at infinity forces it to vanish identically, thereby establishing a unique continuation principle at infinity for first-order equations of Cauchy--Riemann type.

In the case when   the potential $V$ is globally bounded, we first establish a quantitative lower bound for (normalized) solutions of \eqref{eqn} near  infinity as follows.

\begin{theorem}\label{hh}
 Suppose $u\in C(\mathbb C^n)$ satisfies  $u(0)=1$  and $$|u(z)|\le e^{C_0|z|}, \quad \ \  |z|\gg1 $$ for some $C_0>0$, and 
  $$\bar\partial u = Vu\ \  \text{on} \ \ \mathbb C^n$$ in the sense of distributions for  some   $(0,1)$-form $V\in L^\infty(\mathbb C^n)$.  
    Then there exists a positive constant $C$ depending only on $n$,    $C_0$  and  $\|V\|_{L^\infty(\mathbb C^n)} $  such that 
    $$  \inf_{|z_0|=R}\sup_{|z-z_0|<1} |u(z)|\ge e^{-CR\ln R},\ \ R\gg1.$$
\end{theorem}

The case $n=1$ in Theorem \ref{hh} is due to Kenig, Silvestre, and Wang \cite{KSW}, who employed   the three-circle theorem together with a scaling argument of Bourgain and Kenig \cite{BK}. Extending this approach to higher dimensions requires overcoming the compatibility obstruction for the $\bar\partial$-equation. The key observation is that the potential is  $\bar\partial$-closed, making use of a result of Gong and Rosay \cite{GR}. Combined with sup-norm estimates of $\bar\partial$ and a higher-dimensional three-circle theorem, this yields the required quantitative lower bound. 

The $R\ln R$ dependence in the exponent of   Theorem~\ref{hh} is sharp up to a multiplicative constant, as demonstrated in Remark~\ref{lb} by a family of holomorphic functions for which this dependence cannot be improved. 
As a direct consequence, any solution to \eqref{eqn} that decays faster than $e^{-C|z|\ln|z|}$  must be trivial; see Corollary \ref{lan}. 

However, the rate in Corollary \ref{lan} is not optimal for the purpose of   unique continuation. Indeed,  Theorem \ref{ol} below improves upon this conclusion  by showing that sufficiently rapid exponential decay already forces triviality.   Moreover, rather than requiring the equation \eqref{eqn} on all of $\mathbb C^n$, we assume only that  it holds  on an exterior domain, for instance,   $\mathbb C^n\setminus \overline{B_R}$, where $B_R$ denotes the open ball  centered at $0$ with radius $R>0$  in $\mathbb C^n$.    The proof proceeds by slicing along a coordinate axis and then applying the similarity principle and the monotonicity of the winding number. 

\begin{theorem}\label{ol}
Suppose $u\in L^1_{{loc}}(\mathbb C^n\setminus \overline{B_R})$   satisfies
\[
    \bar\partial u=Vu\ \ \text{on}\ \ \mathbb C^n\setminus \overline{B_R}
\]
in the sense of distributions for  some   $(0,1)$-form $V\in L^\infty(\mathbb C^n\setminus \overline{B_R})$.  If   
\begin{equation}\label{dw}
     |u(z)|
    \le
    e^{-C|z|},\qquad |z|\gg1 
\end{equation}
for some $C>2\|V\|_{L^\infty(\mathbb C^n\setminus \overline{B_R})}$, then $u $  vanishes identically.
\end{theorem}

The linear exponential scale and the exponential constant in Theorem~\ref{ol} are both optimal.   Indeed, for any $c>0$, consider the nontrivial smooth function   $u = e^{-c|z|}$ on  $\mathbb C^n\setminus \overline{B_1} $. Then $u\in L^1(\mathbb C^n\setminus \overline{B_1})$ and satisfies  $ \bar\partial u = Vu$ on   $\mathbb C^n\setminus \overline{B_1}, $   where $ V = -\sum_{j=1}^n\frac{cz_j}{2|z|}\,d\bar z_j \in L^\infty (\mathbb C^n \setminus \overline{B_1}) $ with $\|V\|_{L^\infty(\mathbb C^n\setminus \overline{B_1})}=\frac{c}{2}$, showing that the threshold $C>2\|V\|_{L^\infty(\mathbb C^n\setminus \overline{B_R})} $ is sharp.

More generally, we establish a unified unique continuation result in Theorem \ref{gen} for a broad class of radially decaying bounded potentials. The decay required of the solution is determined explicitly by the decay of the potential at infinity, and the resulting threshold is sharp. This framework includes, for example, both the bounded-potential case of Theorem \ref{ol} and logarithmically decaying potentials; see Example \ref{ex11}.

We next consider the unique continuation problem at infinity for solutions satisfying the much weaker condition of $L^2$-flatness \eqref{flatn}. 
  The following theorem establishes unique continuation   for $W^{1,2}$  solutions of  \eqref{eqn}  with $  L^2$ potentials under this assumption. Here, for a positive integer $k$ and $p\ge1$, $W^{k,p}(\Omega)$ denotes the space of functions whose weak derivatives up to order $k$ belong to $L^p(\Omega)$. The result also applies to vector-valued solutions $u=(u_1,\ldots,u_N)^T$, with $u\in W^{k,p}(\Omega)$ understood componentwise. See also Corollary \ref{non} for an application of Theorem \ref{main4n}  to    the uniqueness of    certain nonlinear equations involving the Laplacian  under  sufficiently rapid decay at infinity.

\begin{theorem}\label{main4n} 
Suppose
$u=(u_1,\ldots,u_N)^T
\in W^{1,2}(\mathbb C^n\setminus\overline{B_R})$ 
satisfies
\[
\bar\partial u=Vu
\ \ \text{on }\ \ \mathbb C^n\setminus\overline{B_R}
\]
for some $N\times N$ matrix-valued $(0,1)$-form $V\in L^2(\mathbb C^n\setminus\overline{B_R}). $ 
Assume that $u$ vanishes to infinite order at infinity in the
$L^2$ sense, namely,
\begin{equation}\label{flatn}
\lim_{r\to\infty}
r^m\int_{|z|>r}|u(z)|^2\,dv_z=0
\qquad\text{for every }m>0.
\end{equation}
Then $u$ vanishes identically.
\end{theorem}

  When $n=1$, as shown in the proof of Theorem \ref{main4n}, the problem at infinity can be transformed by inversion into a unique continuation problem at a finite point, with the $L^2$ norm of the potential preserved. In this case, $L^2$ is the critical integrability class: for every $p>2$, there exist potentials in $L^p\setminus L^2$ for which unique continuation fails. In higher dimensions,  however, the situation  differs substantially from its finite-point counterpart. At a finite point, the critical exponent $2n$ is intrinsic to the problem, as reflected by the counterexamples in Remark~\ref{ex4}; see also \cite{PZ3}. From the viewpoint of the proof, the reduction to one dimension is carried out by complex radial slicing, which leads naturally to this dimension-dependent integrability exponent. By contrast, at infinity one may use coordinate slicing together with Fubini's theorem, so that $L^2$ integrability of the potential is preserved on almost every slice. The one-dimensional unique continuation result   can then be applied on almost every such slice, and the resulting vanishing is propagated to the full exterior domain by weak unique continuation. Thus the $L^2$  condition on the potential  in Theorem~\ref{main4n} arises naturally from the one-dimensional sliced problem.  
  
  As shown  by Example \ref{ex3}, unique continuation under the $L^2$-flatness assumption  may fail for potentials outside $L^2$ at infinity. Nevertheless, Theorem \ref{main3} identifies an important class extending beyond $L^2$ class  for which unique continuation remains valid, namely,  potentials satisfying
\(V(z)=O\left(\frac{1}{|z|}\right) 
\) at infinity.

 Furthermore, we consider compactly supported potentials on $\mathbb C^n$. In one complex dimension,    Chirka and Rosay \cite{CR} established the following uniqueness result for the $\bar\partial$-equation: if  $V\in L^\infty(\mathbb C)$ has compact support and  $u\in C^1(\mathbb C)$ satisfies $  \bar\partial u = V u$ on $\mathbb C $ and  $ \lim_{z\rightarrow \infty}u =0, $   then $u\equiv 0$.   Theorem \ref{cr}  below extends this result to higher dimensions while relaxing the assumption  $V\in L^\infty(\mathbb C)$ to $V\in L^p(\mathbb C^n)$ for some $p>2$. 

 \begin{theorem}\label{cr}
 Suppose $u \in L_{loc}^{2}(\mathbb C^n)$   satisfies \begin{equation*} 
        \bar\partial u = Vu \ \ \text{on}\ \  \mathbb C^n 
    \end{equation*}  in the sense of distributions  for some   $(0,1)$-form $V\in L^p(\mathbb C^n), p>2$.    If $V$ has compact support and  $$\lim_{z\rightarrow \infty} u(z) =0,$$ then $u$ vanishes identically. 
 \end{theorem}

Note that compact support of potentials alone is not sufficient for unique continuation. Indeed, Mandache's example shows that weak unique continuation may fail for compactly supported potentials below the $L^2$ integrability threshold; see Remark~\ref{ex4}.   We also recall that, in Meshkov’s counterexample \cite{Me} to Landis’s conjecture  for \eqref{la} on $\mathbb R^2$, the complex-valued  potential vanishes on a sequence of concentric annuli with radii tending to infinity. On the other hand, holomorphic functions clearly fall within the framework of Theorem~\ref{cr} by taking $V\equiv0$. When restricted in this special case, the conclusion   also follows from the classical Liouville theorem, since  boundedness  already forces $u$ to be constant. Consequently, if in addition $u(z)\to0$ as $|z|\to\infty$, then necessarily $u\equiv0$.

 \medskip
 
  Finally, it is worth  pointing out that the conclusions of Theorems~\ref{ol} and~\ref{main4n}, although stated on complements of closed balls, remain valid for any domain that is the complement of a compact set. This follows directly  from the weak unique continuation property in Theorem \ref{pz}. By contrast, the global nature of Theorems \ref{hh} and \ref{cr} is essential, as their  conclusions may fail  when  the domain $\mathbb C^n$ is replaced by the complement of a compact set. This is already evident in one complex dimension. Indeed, for any $k_0\in \mathbb N$, the nontrivial function  $ u(z)= z^{-k_0}, |z|>1$ satisfies  $\bar\partial u =0 =Vu $ for $ |z|>1$ 
with $V\equiv 0$, and  $\lim_{z\rightarrow \infty}u(z)= 0$, showing that Theorem \ref{cr} fails in this setting. See also Remark \ref{gd} for an example on the complement of the unit disc showing that Theorem~\ref{hh} does not extend to exterior domains.

\section{Preliminaries}
 
In this section, we collect several preliminary results concerning unique continuation at a finite point, Sobolev properties for the $\bar\partial$ operator, one-dimensional inversion and coordinate slicing.  These results will be used repeatedly in the proofs of the main theorems.

We first recall a unique continuation result for the $\bar\partial$ operator at a finite point under $L^2$-flatness. A function $u\in L_{ {loc}}^2(\Omega)$ is said to vanish to infinite order, or to be $L^2$-flat, at a point $z_0\in\Omega$ if 
$$  \lim_{r\rightarrow 0} r^{-m}\int_{|z-z_0|<r}|u(z)|^2\  dv_z = 0
\qquad\text{for every }m>0. $$ 

\begin{theorem}\cite{PZ}\cite{PZ3}\label{pz}
Let $\Omega$ be a  domain in $\mathbb C^n$. Suppose $u=(u_1, \ldots, u_N)^T \in W^{1, 2}_{loc}(\Omega)$      satisfies $ |\bar\partial u|\le V|u|$ almost everywhere\ on $\Omega$ for some nonnegative measurable function $V$ on $\Omega$.  
 \begin{enumerate}
    \item  The weak unique continuation holds if $V\in L_{loc}^2(\Omega)$: if $u $ vanishes  in an open subset of $ \Omega$,   then $u$ vanishes identically. 
 \item  The (strong) unique continuation holds if  $V\in L_{loc}^2(\Omega)$ when $n=1$, or  if $V \in L_{loc}^{p}(\Omega)$ for some $p>2n$ when $n\ge 2$: if $u$ vanishes to infinite order in the $L^2$ sense at some $z_0\in \Omega$, then $u$ vanishes identically.

 \item  Suppose that $0\in \Omega$ and  $  V  \le \frac{C}{|z|}$ on $\Omega$  for some   $C>0$. Then   the (strong) unique continuation property holds if either  $N =1$, or  $N\ge 2$ and $C< \frac{1}{4 }$: if $u$  vanishes to infinite order in the $L^2$ sense at $0$, then $u $ vanishes identically.
 \end{enumerate}

 \end{theorem}

  We note that   every function $u=(u_1, \ldots, u_N)^T\in W^{1,1}_{loc}(\Omega)$ satisfying $$ |\bar\partial u| \le V|u|\ \ \text{a.e. on}\ \ \Omega$$ with  $V \in L^p_{loc}(\Omega)$  can be viewed as    a weak solution of  a first order  Schr\"odinger-type system for  $\bar\partial$. Indeed, fix a measurable representative of $u$ and let $Z=\{u=0\}$. The set $Z$ is well defined up to a null set. By the standard Stampacchia property that the weak gradient of a Sobolev function vanishes almost everywhere on each of its level sets, $\bar\partial u=0$ almost everywhere on $Z$.  Define the \(N\times N\) matrix of \((0,1)\)-forms $\mathcal V=(\mathcal V_{kj}) $ with \(\mathcal V_{kj}=0\) on $Z$ and   $\mathcal V_{kj}:=
\frac{(\bar\partial u_k)\overline{u_j}}{|u|^2}$ otherwise. Then \(\mathcal V_{kj}\in L^p_{ {loc}}(\Omega )\), and \[ \bar\partial u=\mathcal V u \quad \text{a.e. on } \ \Omega. \]   On the other hand, if $f\in L_{loc}^1(\Omega)$ and $u\in W_{loc}^{1,1}(\Omega)$, then  
$$ \bar\partial u = f\qquad \text{on}\ \ \Omega$$
in the sense of distributions is equivalent to 
$$\bar\partial u = f\qquad \text{a.e. on}\ \ \Omega.  $$
Thus, results for the system $\bar\partial u=\mathcal V u$ apply
equally to the differential inequality
$|\bar\partial u|\le V|u|$ through the above reduction. We shall also
freely pass between the almost-everywhere and distributional
formulations whenever the relevant quantities belong to
$L_{ {loc}}^1$. For more results concerning  unique continuation   for the inequality $|\bar\partial u| \le V|u|$ at a finite point, see, for instance, \cite{PZ, PZ2, PZ3, Shi}. 


\begin{remark} \label{ex4} The unique continuation    fails in general   in the following situations.  
    \begin{enumerate}
\item The weak unique continuation may fail if the potential does not belong to $L_{loc}^{2}$. Indeed, by  an example of Mandache \cite{Ma02}, for every $0<p<2,$ there exist a nontrivial smooth   function  $u$ on $\mathbb C$, supported in the unit disc, and a $(0,1)$-form  $V\in L^p(\mathbb C) $ such that  $ \bar\partial u =  Vu$ on $\mathbb C$.

        \item  The strong unique continuation   may fail if the potential does not belong to $L_{loc}^{2n}$. Indeed,   
for each $1\le p<2n$, choose $\epsilon  \in (0, \frac{2n}{p}-1)$. The   function $u= e^{-\frac{1}{|z|^\epsilon}}$ extended by $u(0)=0$ is smooth and   vanishes to infinite order  at $0$. Moreover, 
$  |\bar\partial u| \le V|u|: = \frac{\epsilon }{2 |z|^{\epsilon+1}}|u|$ on $ B_1\subset \mathbb C^n.$  By the choice of $\epsilon$, we have  $ V \in L^p(B_1)$.

 \item The   unique continuation  property  in Theorem \ref{pz} part 3 may also fail  in general if $N\ge 2$ and  $C$ is sufficiently large,  as shown by   examples  of the first author and  Wolff  \cite{PW98}, and of  Alinhac and Baouendi \cite{AB94}; see also \cite[Example 5]{PZ2}. More recently,  Chen, Fan and Tang \cite{CFT} showed that $C=\frac{1}{4}$ is the sharp threshold for unique continuation.  More precisely, they constructed  a nontrivial smooth  function that vanishes to infinite order  at $0$ and satisfies  $ |\bar\partial u |\le \frac{|u|}{4|z|} $ almost everywhere on $\mathbb C^n$.
     \end{enumerate}
\end{remark}
  
\medskip

Regarding regularity,    $\bar\partial$ is a first-order  elliptic operator and improves local Sobolev $W^{k,p}, 1<p<\infty$ regularity precisely by one order; see \cite[Lemma 3.1]{PZ3}. Combining this property  with a bootstrap argument yields the following regularity result for solutions of \eqref{eqn} with $L^p$ potentials, $p>2n$.  
 Throughout the paper, whenever a Sobolev function admits a continuous representative. 
we always identify the Sobolev equivalence class with this representative. All subsequent pointwise statements concerning such functions are understood in this sense.

\begin{lem}\label{bs}
Let \(\Omega\) be a domain in \(\mathbb C^n\), \(n\ge 1\), and let
\(u=(u_1,\ldots,u_N)^T\) satisfy
\[
\bar\partial u=Vu\qquad \text{on}\ \ \Omega
\]
in the sense of distributions for some measurable matrix-valued 
\((0,1)\)-form $V$.

\begin{enumerate}
\item If \(u\in L^2_{ {loc}}(\Omega)\) and
\(V\in L^p_{ {loc}}(\Omega)\) for some \(2n<p<\infty\), then
 \(
u\in W^{1,p}_{ {loc}}(\Omega).
\)

\item If \(u\in L^1_{ {loc}}(\Omega)\) and
\(V\in L^\infty_{ {loc}}(\Omega)\), then
\(u\in W^{1,q}_{ {loc}}(\Omega)\) 
for every $1\le q<\infty.$ 
 
\end{enumerate}
In particular, in either case we have \(u\in C(\Omega)\).
\end{lem}

\begin{proof}
We first prove part 1. By H\"older's inequality,  $Vu\in L_{loc}^\frac{2p}{p+2}(\Omega)$ with $ \frac{2p}{p+2}> 1$. Since the $\bar\partial$ operator  improves local regularity precisely by one, it follows that $u\in W^{1, \frac{2p}{p+2}  }_{loc}(\Omega)$.   By the Sobolev embedding theorem, $u\in L_{loc}^{ \frac{2np}{np+2n-p}}(\Omega)$. A standard bootstrap argument eventually gives $u\in L_{loc}^{\tilde p}(\Omega)$ for some $\tilde p> \frac{2np}{p-2n}$. Consequently, $Vu\in L_{loc}^{q}(\Omega)$ for some  $q>2n$.  Elliptic regularity then gives $u\in W^{1, q}_{loc}(\Omega)$.  Since $q>2n$, the Sobolev embedding theorem also gives    $u\in C(\Omega)$ and hence $Vu\in L^{p}_{loc}(\Omega)$.   A final application of the local elliptic regularity of 
$\bar\partial$ yields $u\in W_{ {loc}}^{1,p}(\Omega).$ 

For  part 2, since \(V\in L^\infty_{ {loc}}(\Omega)\),
we have \(Vu\in L^1_{ {loc}}(\Omega)\). The local \(L^p\) estimates
for \(\bar\partial\) then gives \(u\in L^{q_0}_{ {loc}}(\Omega)\) for some
\(q_0>1\). Indeed, this follows from the $L^p$ estimates of the solution operator \(T_q\) of $\bar\partial$ in \cite[p.~86, Proposition 4.24]{LM}, applied with \(q=0\) and \(p=1\), together with the fact that
\( u-T_0(Vu) \)
is holomorphic. Consequently, 
\(Vu\in L^{q_0}_{\mathrm{loc}}(\Omega)\). Applying the same bootstrap argument
as above, we obtain 
\(u\in W^{1,q}_{ {loc}}(\Omega)\) for every \(q<\infty\). Hence
\(u\in C(\Omega)\).
\end{proof}

Although the chain rule and/or product rule do not hold in general for Sobolev functions without additional assumptions, we identify a class of Sobolev functions for which they are valid. Since this result will be used repeatedly throughout the paper, we provide a proof below.  

\begin{lem}\label{pc}
    Let $\Omega$ be a domain in $\mathbb R^d, d\ge 2$. Suppose that  $u\in W^{1, p}_{loc}(\Omega)$ and $\lambda\in W^{1, q}_{loc}(\Omega)$ for some $p\ge 1$ and   $q>d$. Then $ ue^\lambda\in W^{1, s}_{loc}(\Omega)$ with $s=\min\{p, q\}$, and 
    \begin{equation}\label{cp}
        \nabla (ue^\lambda) = e^\lambda\nabla u  + ue^\lambda \nabla \lambda \ \ \text{on}\ \ \Omega
    \end{equation} 
    in the sense of distributions.
\end{lem}

\begin{proof}
   Since $\lambda\in W^{1, q}_{loc}(\Omega)$ for some   $q>d$, the Sobolev embedding theorem implies that $\lambda\in C(\Omega)$. In particular, $\lambda$ and  $e^\lambda$ are  locally bounded in $\Omega$. Thus  by the Sobolev chain rule (see, for instance, \cite[ p. 48]{Zie}), $e^\lambda\in W^{1, q}_{loc}(\Omega)$ and
  \begin{equation}\label{ew}
      \nabla e^{ \lambda} =  e^{ \lambda} \nabla \lambda \ \ \text{on}\ \ \Omega
  \end{equation} 
  in the sense of distributions.    

 Let $\phi\in C_c^\infty(\Omega)$ and $U\Subset \Omega$ be a bounded smooth domain  containing $ supp \ \phi$. Let $u_j(\in C^\infty(U))\rightarrow u$ in $W^{1, p}(U)$ norm.  
 Then 
 \begin{equation*}
     \begin{split}
        - \langle ue^{\lambda}, \nabla \phi \rangle =   \lim_{j\rightarrow \infty} - \langle u_je^{\lambda}, \nabla \phi \rangle =   \lim_{j\rightarrow \infty} - \langle e^{\lambda}, (\nabla \phi) u_j\rangle  = \lim_{j\rightarrow \infty} - \langle e^{\lambda}, \nabla( \phi u_j)\rangle +  \lim_{j\rightarrow \infty} \langle e^{\lambda}, \phi\nabla u_j\rangle
     \end{split}
 \end{equation*}
Since  $\phi u_j\in C_c^\infty(\Omega)$, one deduces from \eqref{ew} that    $$ \langle e^{\lambda}, \nabla( \phi u_j)\rangle =  -\langle e^{\lambda}\nabla \lambda,  \phi u_j\rangle.$$  Making use of  the boundedness of $e^{\lambda}$ on $U$ and the fact that $ u_j\rightarrow u $ in $ L^{\frac{d}{d-1}}(U)$ norm due to the continuous embedding of $W^{1, p}(U)$ into $ L^{\frac{d}{d-1}}(U)$, we have
\begin{equation*} 
     \lim_{j\rightarrow \infty} - \langle e^{ \lambda}, \nabla( \phi u_j)\rangle = \lim_{j\rightarrow \infty}  \langle e^{ \lambda} \nabla \lambda ,  \phi u_j\rangle =  \langle u e^{ \lambda}\nabla \lambda,  \phi \rangle.
\end{equation*}
Similarly, 
$$\lim_{j\rightarrow \infty} \langle e^{ \lambda}, \phi\nabla u_j\rangle =\langle  e^{ \lambda}\nabla u,   \phi \rangle.  $$ 
 Altogether, we obtain the desired equality \eqref{cp}.

 It remains to verify the claimed regularity. If $p\ge q$, then
$p\ge q>d$, so the Sobolev embedding theorem gives
$u\in L_{  {loc}}^\infty(\Omega)$. Hence both terms on the
right-hand side of \eqref{cp} belong to $L_{  {loc}}^q(\Omega)$.
Suppose instead that $p<q$. The Sobolev embedding theorem gives
$u\in L_{  {loc}}^{\frac{dp}{d-p}}(\Omega)$ if $p<d$, and $u\in    L_{loc}^{ \tilde p}(\Omega)$ for all $\tilde p<\infty$ otherwise. Since $q>d$, it follows that $u \in L_{  {loc}}^{\frac{qp}{q-p}}(\Omega)$. H\"older's inequality yields
 $u\nabla\lambda\in L_{  {loc}}^p(\Omega).$ 
Since $e^\lambda$ is locally bounded, both terms on the right-hand
side of \eqref{cp} belong to $L_{  {loc}}^p(\Omega)$. Therefore,  in either case, $ue^\lambda\in   W^{1, s}_{loc}(\Omega)$.
 \end{proof}

We shall also use two preparatory lemmas. The first concerns the one-dimensional inversion that transforms a unique continuation problem at infinity into a local problem near a finite point, while the second is a coordinate-slicing lemma used to reduce higher-dimensional problems to one complex dimension.

\begin{lem}\label{rev}
 Suppose $u=(u_1, \ldots, u_N)^T\in W^{1, p}_{loc}(\mathbb C\setminus \overline{D_R}),\  p\ge 1 $ satisfies 
 $$ \bar\partial u = Vu\ \ \text{a.e. on}\ \ \mathbb C\setminus \overline{D_R} $$
  for some matrix-valued $(0,1)$-form $V\in L^2_{loc}(\mathbb C\setminus \overline{D_R})$. Define $ v(w): = u(\frac{1}{w}) $ and $W(w): = -\frac{V(\frac{1}{w})}{ \bar w^2}, w\in D_{\frac{1}{R}}\setminus\{0\}$. Then $v\in W^{1, p}_{loc}(  {D_\frac{1}{R}}\setminus \{0\})$ and $W\in L^2_{loc}({D_\frac{1}{R}}\setminus \{0\})$, with 
$$\bar\partial v = Wv   \ \ \text{a.e. on}\ \   D_{\frac{1}{R}}\setminus\{0\}.$$   \end{lem}

\begin{proof}
The inversion $ F(w)=\frac1w$ is a smooth diffeomorphism from
$D_{\frac{1}{R}}\setminus\{0\}$ onto
$\mathbb C\setminus\overline{D_R}$ and is bi-Lipschitz on compact
subsets away from $0$. Hence  $W\in L_{ {loc}}^2
(D_{\frac{1}{R}}\setminus\{0\}).$  Moreover, the standard chain rule for Sobolev functions (see \cite[ p. 52]{Zie}) gives
\[
v=u\circ F
\in W_{{loc}}^{1,p}
(D_{\frac{1}{R}}\setminus\{0\}),
\]
and
$$
\bar\partial_w v(w)
=
-\frac{1}{\bar w^2}
(\bar\partial_z u)\left(\frac1w\right) = W(w)v(w)\ \ \text{a.e. on}\ \   D_{\frac{1}{R}}\setminus\{0\}.
$$
\end{proof}

\begin{lem}\label{slice}
Let $\Omega\subset\mathbb C^n$, $n\ge2$, be a domain, and write
$z=(z_1,z')\in\mathbb C\times\mathbb C^{n-1}$. For $z'\in\mathbb C^{n-1}$, set $ \Omega_{z'}:=\{\zeta\in\mathbb C:(\zeta,z')\in\Omega\}. $
Suppose that $u\in L^p_{  loc}(\Omega),\ f\in L^q_{  loc}(\Omega)$ for some $p, q\ge 1$ and that
\[
\bar\partial_{z_1}u=f\ \ \text{on}\ \ \Omega
\]
in the sense of distributions. Then, for almost every $z'$, $u(\,\cdot\,,z')\in L^p_{  loc}(\Omega_{z'}),\,f(\,\cdot\,,z')\in L^q_{  loc}(\Omega_{z'}) $
and
\begin{equation}\label{se}
    \bar\partial_\zeta u(\zeta,z') 
=
f(\zeta,z') \ \ \text{on}\ \ \Omega_{z'}
\end{equation}
  in the sense of distributions. If, in addition, $u\in W^{1,p}_{  loc}(\Omega)$, then $ u(\,\cdot\,,z')\in W^{1,p}_{  loc}(\Omega_{z'})$ for almost every $z'$.
\end{lem}

\begin{proof}
It suffices to argue locally. Let $D\subset\mathbb C$ and
$U\subset\mathbb C^{n-1}$ be relatively compact open sets such that $ D\times U\Subset\Omega.$ 
For $\varphi\in C_c^\infty(D)$ and $\psi\in C_c^\infty(U)$, testing the
distributional identity $\bar\partial_{z_1}u=f$ against
$\varphi(\zeta)\psi(z')$ gives
\[
\int_U \psi(z')
\left[
\int_D f(\zeta,z')\varphi(\zeta)\,dv_\zeta
+
\int_D u(\zeta,z')\bar\partial_\zeta\varphi(\zeta)\,dv_\zeta
\right]dv_{z'}=0.
\]
By Fubini's theorem, the expression in brackets vanishes for almost every
$z'\in U$, for each fixed $\varphi$. A countable dense family of test functions in $C_c^\infty(D)$ allows the exceptional null set in $z'$ to be chosen independently of $\varphi$. Hence, \eqref{se} holds on $D$ in the sense of distributions for
almost every $z'\in U$. 

Moreover, Fubini's theorem gives $u(\,\cdot\,,z')\in L^p(D), \ 
f(\,\cdot\,,z')\in L^q(D)$ for almost every $z'\in U$.  
Finally, if $u\in W^{1,p}_{   loc}(\Omega)$, then the standard slicing property of Sobolev functions yields  $u(\,\cdot\,,z')\in W^{1,p}_{  loc}(\Omega_{z'})$. 
\end{proof}

\section{Proof of Theorem \ref{hh}}
 In this section, we prove  Theorem \ref{hh}, which provides  a quantitative   lower bound for the sup-norm of solutions to \eqref{eqn} near infinity.  
We note that, by  Lemma \ref{bs} part 2 and the boundedness of the potential $V$, the continuity assumption on $u$ in Theorem \ref{hh} may be weakened to $u\in L^1_{ {loc}}(\mathbb C^n)$. We retain the continuity assumption there and below only to ensure that the sup-norm of $u$ appearing in the statement is well defined. As in \cite{KSW} for the case  $n=1$, the main tools of the proof  are the three-circle theorem for holomorphic functions and the scaling argument of Bourgain-Kenig \cite{BK}.

We   begin  with     a  higher-dimensional  version of Hadamard's three-circle theorem for holomorphic functions.

\begin{lem}\label{the}
      Let $h$ be holomorphic in $B_2\subset \mathbb C^n$. Then for any $0<r<r_1<r_2<2$,
    $$ \|h\|_{L^\infty(B_{r_1})}\le \|h\|_{L^\infty(B_{r})}^\theta\|h\|_{L^\infty(B_{r_2})}^{1-\theta},$$
    where $ \theta =\frac{\ln r_2 -\ln r_1}{\ln r_2-\ln r} $.
\end{lem}

\begin{proof}
 For each $\zeta\in S^{2n-1}$, define $h^\zeta(w) : = h(w\zeta), w\in D_2$. Then $h^\zeta$ is holomorphic in $D_2$. 
  By Hadamard’s three-circle theorem 
    $$ \|h^\zeta\|_{L^\infty(D_{r_1})}\le \|h^\zeta\|_{L^\infty(D_{r})}^\theta\|h^\zeta\|_{L^\infty(D_{r_2})}^{1-\theta}.$$
 Taking the supremum over \(\zeta\in S^{2n-1}\) gives
    $$   \sup_{\zeta\in S^{2n-1}}\|h^\zeta\|_{L^\infty(D_{r_1})}\le    \sup_{\zeta\in S^{2n-1}}\|h^\zeta\|_{L^\infty(D_{r})}^\theta   \sup_{\zeta\in S^{2n-1}}\|h^\zeta\|_{L^\infty(D_{r_2})}^{1-\theta}. $$
  The desired estimate then follows from the fact that, for every \(0<\rho<2\),
    $$ \|h\|_{ L^\infty(B_{\rho}) } =   \sup_{\zeta\in S^{2n-1}}\|h^\zeta\|_{L^\infty(D_{\rho})}. $$
\end{proof}

  We  next apply  the three-circle theorem  above to derive a lower bound for the sup-norm of solutions to \eqref{eqn} on   bounded domains. To this end,  we first relate a solution of \eqref{eqn} to a holomorphic function. When $n=1$, this is essentially immediate, since the 
$\bar\partial$-equation has no  compatibility condition.  In higher dimensions, the key observation is  the $\bar\partial$-closedness of the potential $V$, relying on the following result of Gong-Rosay \cite[Proposition A]{GR}. Combined with  sup-norm estimates for the $\bar\partial$-equation on balls, this  allows us to  obtain  an upper bound for the maximal  vanishing order   of solutions to \eqref{eqn} on balls.

\begin{theorem}\cite{GR}\label{gr1}
  Let $\Omega$ be a domain in $\mathbb C^n$. Suppose  $u \in C(\Omega)$    satisfies  $  \bar\partial u  = Vu$ on $\Omega$  in the sense of distributions for some $(0,1)$-form  $V\in L^\infty(\Omega)$. Then the zero set $u^{-1}(0) $ of $u$ is a complex analytic variety.  
  \end{theorem}

\begin{lem}\label{l1}
Let $u \in C( B_2)  $  and $V$ be a   $(0,1)$-form   on $B_2$  such that $$\|V\|_{L^\infty(B_2)}\le M $$  for a  constant $M>0$. Suppose that  \begin{equation}\label{11}
    \|u\|_{L^\infty(B_1)}\ge 1 \ \ \text{and}  \ \ \|u\|_{L^\infty(B_2)}\le e^{C_0M}. 
\end{equation}    
   for a constant  $C_0>0$ and $u$ satisfies $$    \bar\partial u = Vu\ \ \text{on}\ \ B_2$$ in the sense of distributions. Then there exist  two positive constants $C_1$ and $ C_2 $ dependent only on $C_0$ and $n$,    such that for all $0<r<1$, 
\begin{equation*}
       \|u\|_{L^\infty(B_r)}\ge C_1^Mr^{C_2M}.
   \end{equation*}
    \end{lem}

\begin{proof} 
Let $S: = \{z\in B_2: u(z) = 0\}\subsetneq B_2$. Since $u$ is continuous, $B_2\setminus S$ is open. For every
$z_0\in B_2\setminus S$, we may choose  a sufficiently small neighborhood $U$ of $z_0$ such that  $u$ does not vanish on $U$ and a branch of $\log u$ is defined there. By Lemma \ref{bs}, we have $u\in W_{loc}^{1,p}(U)$ for all $p<\infty$.  The Sobolev chain rule \cite[ p. 48]{Zie} then gives $\log u\in W_{loc}^{1,p}(U)$ and 
\[
 \bar\partial(\log u)=\frac{\bar\partial u}{u} =V
\ \ \text{on }\ \  U.
\]
 Hence $\bar\partial V=0$ on $U$ in the sense of distributions.
Therefore, $ V$ is $\bar\partial$-closed on   $B_2\setminus S.$ 
Since  $V\in L^\infty(B_2)$, Theorem \ref{gr1} of Gong and Rosay implies that $S$ is a complex analytic variety in $B_2$.   It then follows from  Demailly's removable singularity result \cite[Lemma 6.9]{De}  that $$\bar\partial V =0\ \  \text{on}\ \  B_2$$ 
in the sense of distributions. 
Consequently,   the well-known sup-norm estimates for the $\bar\partial$-equation on balls provide a  function $\lambda\in L^\infty(B_2)$ such that $$\bar\partial \lambda = V \ \ \text{on}\ \ B_2$$ with
\begin{equation}\label{bd}
    \|\lambda\|_{L^\infty(B_2)}\le  {A(n)}\|V\|_{L^\infty(B_2)} =  {A(n)} M
\end{equation}
where $ {A(n)}>0$ depends only on $n$; see, for instance, \cite[pp. 94-95
]{LM}.  Moreover,  the  ellipticity of $\bar\partial$ implies that $\lambda\in W^{1, p}_{loc}(B_2)$ for all $p<\infty$. 
 
Since $u\in W^{1, p}_{loc}(B_2)$ for all $p<\infty$ by Lemma \ref{bs}, we can apply Lemma \ref{pc} to obtain  
$$ \bar\partial(ue^{-\lambda}) = \bar\partial u e^{-\lambda} - ue^{-\lambda} \bar\partial \lambda = Vu e^{-\lambda} -Vu e^{-\lambda} =0\ \ \text{on}\ \ B_2.  $$
  Thus $$ h: = ue^{-\lambda} $$ is holomorphic on $B_2$. By \eqref{11} and  \eqref{bd},  we have
   \begin{equation}\label{331}
   \begin{split}
        &\|h\|_{L^\infty(B_{1})}\ge e^{- {A(n)}M}\|u\|_{L^\infty(B_{1})}\ge e^{- {A(n)}M}\ge e^{-\tilde CM}; \\
        &\|h\|_{L^\infty(B_{2})}\le e^{ {A(n)}M}\|u\|_{L^\infty(B_{2})}\le e^{ ( {A(n)}+C_0)M}\le e^{\tilde CM}, 
   \end{split}
        \end{equation}where $\tilde C: = {A(n)}+C_0$.

 Applying Lemma \ref{the} with   \(r_1=1\) and $r_2 = \frac{3}{2}$, we obtain  from \eqref{331} that for all $0<r<1$,  
   \begin{equation*}
       e^{-\tilde CM}\le  \|h\|_{L^\infty(B_{1})} \le  \|h\|_{L^\infty(B_{r})}^\theta\|h\|_{L^\infty(B_{3/2})}^{ 1-\theta } \le  \|h\|_{L^\infty(B_{r})}^\theta e^{(1-\theta)\tilde CM},
   \end{equation*}
   where    $\theta =\frac{\ln 3-\ln 2}{\ln 3-\ln 2 -\ln r}$. Hence
   \begin{equation*}
       \|h\|_{L^\infty(B_{r})}\ge e^{-\frac{2\tilde CM}{\theta}}e^{\tilde CM}.
   \end{equation*}
Using  \eqref{bd}, $u =he^\lambda$ and   $\tilde C> {A(n)}$, we obtain
   \begin{equation*}
       \|u\|_{L^\infty(B_r)}\ge e^{-A(n)M} \|h\|_{L^\infty(B_r)} \ge e^{-\frac{2\tilde CM}{\theta}}=    C_1^M r^{C_2M},
   \end{equation*}
   where $C_1 =  e^{-2\tilde C}  $ and $C_2 = \frac{2\tilde C}{\ln 3-\ln 2}$. 
Both constants depend only on $C_0$ and $n$.
\end{proof}
\medskip

It is worth noting that  the continuity of $u$ is essential in order to conclude the $\bar\partial$-closedness of the $(0,1)$-form  $\frac{\bar\partial u}{u}$ in the proof of Lemma \ref{l1}. Indeed, the following example shows that, without continuity, this form may be \(\bar\partial\)-closed off a closed set but fail to remain \(\bar\partial\)-closed across it.
  
\begin{example}
  Let \[
F=\{z\in \mathbb C^2: \operatorname{Re} z_1=0\}. 
\]
Then $F$ is a closed real hypersurface in $\mathbb C^2$ of  real codimension \(1\).
Define $$u_0(z)=
\begin{cases}
1, & \operatorname{Re} z_1<0,\\
e^{\bar z_2}, & \operatorname{Re} z_1>0.
\end{cases}$$
Then $u_0$ is discontinuous across $F$ and smooth on $\mathbb C^2\setminus F$. On $\mathbb C^2\setminus F$, we have
\[\frac{\bar\partial u_0}{u_0} = 
H(\operatorname{Re} z_1)\,d\bar z_2, 
\]
 where \(H\) denotes the Heaviside function. Thus $\frac{\bar\partial u_0}{u_0}$ is a bounded $(0,1)$-form and is \(\bar\partial\)-closed off \(F\). However, in the sense of distributions,
\[
\bar\partial \left(\frac{\bar\partial u_0}{u_0}\right)
=
\frac12\,\delta_{\{\operatorname{Re} z_1=0\}}\,
d\bar z_1\wedge d\bar z_2
\neq 0.
\]
Therefore \(\frac{\bar\partial u_0}{u_0}  \) is not \(\bar\partial\)-closed across \(F\).
\end{example}

Finally, the scaling method of Bourgain-Kenig \cite{BK} converts the preceding  lower bound for solutions of \eqref{eqn} on bounded domains into a corresponding lower bound near infinity as stated in Theorem \ref{hh}.

\begin{proof}[Proof of Theorem \ref{hh}: ]
We first prove the theorem under the normalization $\|V\|_{L^\infty(\mathbb C^n)}\le 1$.     Fix $z_0\in \mathbb C^n$ with $|z_0| =R\gg1$, and define  $u_R(z): = u(Rz+z_0)$ and $V_R= R\cdot V(Rz+z_0)$ on $B_2$. Then   for all sufficiently large $R$,
    \begin{equation*}\begin{split}
        &\|u_R\|_{ L^\infty(  B_2)}\le e^{C_0(2R+|z_0|)}\le e^{3C_0R};\\
        &\|V_R\|_{L^\infty(  B_2)}\le R, 
    \end{split}
          \end{equation*}
    and
    \begin{equation*}
        \bar\partial u_R(z) = R\cdot V(Rz+z_0)u_R(z) =V_R(z)u_R(z),\ \ z\in B_2.
    \end{equation*}
    Since $u_R(-\frac{z_0}{R}) = u(0) =1$, 
    $$ \|u_R\|_{ L^\infty(  B_1)  }\ge  u_R\left(-\frac{z_0}{R}\right) =1.$$
    Applying Lemma \ref{l1}, with $M=R$ and $C_0$ there replaced by
$3C_0$, we obtain constants $C_1,C_2>0$, depending only on $C_0$ and $n$,  such that for $0<r<1$, 
    \begin{equation*}
        \|u_R\|_{L^\infty(  B_r)  }\ge C_1^Rr^{C_2R}. 
    \end{equation*}
   Let $R_0>1$ be such that $C_1\ge R_0^{- {C_2} } $.   Furthermore, for every $z_0\in\mathbb C^n$ with $|z_0|=R\ge R_0$,
taking $r=R^{-1}$ gives  
    \begin{equation*}
       \sup_{|z-z_0|<1} |u | = \|u_R\|_{L^\infty(  B_{R^{-1}})  }  \ge C_1^R R^{-C_2R}\ge R^{- 2C_2 R}  =   e^{-2C_2R\ln R}.
    \end{equation*}
     Thus the desired lower bound follows  in the case
\(\|V\|_{L^\infty(\mathbb C^n)}\le 1\), with $C=2C_2$ depending only on $C_0$ and $n$.
    
It  remains to remove the normalization on $V$.  Let $M: = \|V\|_{L^\infty(\mathbb C^n)} >1$ and define $v(z): = u\left(\frac{z}{ M}\right), z\in \mathbb C^n$. Then $ v(0)=u(0)= 1$, $$ |v(z)| = \left|u\left(\frac{z}{ M}\right)\right|\le e^{\frac{C_0}{ M}|z|}\le e^{C_0|z|}, \qquad |z|\gg1,$$
    and $$\bar\partial v(z) = \frac{1}{M}V\left(\frac{z}{ M}\right)u\left(\frac{z}{ M}\right) = \tilde V(z) v(z)\qquad \text{on}\ \ \mathbb C^n,$$
with $\tilde V: = \frac{1}{M}V\left(\frac{\cdot}{ M}\right)$. In particular,  $\|\tilde V\|_{L^\infty(\mathbb C^n)}=1$. Applying the normalized case to $v$, we obtain
$$ \inf_{|z_0|=R}\sup_{|z -z_0|<1} |v(z)|\ge e^{-CR\ln R},\ \ R\gg1,$$
where $C$ depends only on $C_0$ and $n$.
Consequently, since $M>1$, 
\begin{equation*}
\begin{split}
    \inf_{|z_0|=R}\sup_{|z-z_0|<1} |u(z)|   =&     \inf_{|z_0|=R}\sup_{|z -Mz_0|<M} |v(z)|
    \ge  \inf_{|z_0|=R}\sup_{|z -Mz_0|<1} |v(z)|\\ =   &  \inf_{|z_0|=MR}\sup_{|z -z_0|<1} |v(z)|  \ge e^{-CMR\ln( MR)}\ge e^{-\tilde C R\ln R},\ \ R\gg1,
\end{split}
    \end{equation*}  
 for some $\tilde C>0$ depending only on $C_0$, $n$ and $\|V\|_{L^\infty(\mathbb C^n)}$.   The proof is complete. 
    \end{proof}
\medskip

Several remarks concerning the statement of Theorem~\ref{hh} are in order.

\begin{remark}\label{lb}
The $R\ln R$ dependence in the exponent of the lower bound in Theorem~\ref{hh} is optimal  up to a multiplicative constant, even in the holomorphic case with $n=1$ and $V\equiv0$. This is demonstrated by the following family of entire functions. 

Fix $0<c<C_0$. For each sufficiently large $R>1$, choose
$z_R\in\mathbb C$ with $|z_R|=R$, and let
\begin{equation}\label{m}
    m=\lfloor cR\rfloor.
\end{equation}
Define
\[
u_R(z)
=
\left(\frac{z_R-z}{z_R}\right)^m
=
\left(1-\frac{z}{z_R}\right)^m,
\qquad z\in\mathbb C.
\]
The family $\{u_R\}$ is entire, 
$u_R(0)=1$ and  
\[
\bar\partial u_R=0=Vu_R\qquad \text{on}\ \  \mathbb C
\]
with $V\equiv0$. Moreover,  by \eqref{m}
\[
|u_R(z)|
\le
\left(1+\frac{|z|}{R}\right)^m
\le
e^{\frac{m}{R}|z|}
\le
e^{c|z|}
\le
e^{C_0|z|}\qquad \text{on}\ \  \mathbb C.
\]

On the other hand, 
\[
\sup_{|z-z_R|<1}|u_R(z)|
=
\sup_{|z-z_R|<1}
\left|\frac{z-z_R}{z_R}\right|^m
=
R^{-m}.
\]
Consequently,
\[
\inf_{|\zeta|=R}
\sup_{|z-\zeta|<1}|u_R(z)|
\le
\sup_{|z-z_R|<1}|u_R(z)|
=
R^{-m}
=
e^{-m\ln R}\le
e^{-\frac{c}{2}R\ln R}.
\]
Here, the last inequality follows from \eqref{m}, and hence  $m\ge \frac{c}{2}R $ for all sufficiently large $R$. 
Thus, the family $\{u_R\}_{R\gg1}$ satisfies the assumptions of Theorem~\ref{hh} uniformly in $R$ and exhibits decay of order $e^{-cR\ln R}$. 
Consequently, the $R\ln R$ exponent in Theorem~\ref{hh} cannot, in general,
be improved even for entire holomorphic functions.
\end{remark}

\begin{remark}\label{gd}
The conclusion of Theorem~\ref{hh} generally fails if the domain
$\mathbb C^n$ is replaced by the complement of a compact set. This is
already evident in one complex dimension. Let $\Omega=\mathbb C\setminus\overline{D_1}.$ 
For each $m\in\mathbb N$, define
\[
u_m(z):=\left(\frac{2}{z}\right)^m,\qquad z\in\Omega.
\]
Then $u_m$ is holomorphic on $\Omega$, satisfies $u_m(2)=1$, and $$\bar\partial u_m=0=Vu_m\qquad \text{on}\ \ \Omega $$
with $V\equiv0$. Moreover,
\[
|u_m(z)|\le 1\le e^{C_0|z|},\qquad |z|\ge2,
\]
for any $C_0>0$.

On the other hand, for every  $R>2$, and $z_0\in \mathbb C$ with $|z_0|=R$,
\[
\sup_{|z-z_0|<1}|u_m(z)|
=
\left(\frac{2}{R-1}\right)^m.
\]
Thus, given any $C>0$, by choosing $m=\lceil 2CR\rceil$ and then
taking $R$ sufficiently large, we obtain
\[
\sup_{|z-z_0|<1}|u_m(z)|\le
\left(\frac{2}{R-1}\right)^{2CR}
=
e^{-2CR\ln\frac{R-1}{2}} < e^{-CR\ln R}.
\]
Hence no lower bound of the form appearing in Theorem~\ref{hh}, with
a constant depending only on $C_0$ and $\|V\|_{L^\infty}$, can hold
on the complement of a compact set.
\end{remark}

\begin{remark}
The growth assumption $|u(z)|\le e^{C_0|z|}$ for sufficiently large $|z|$  cannot in general be omitted. Indeed, consider 
\[
u(z)=e^{e^{z_1}-1},\qquad z=(z_1,\ldots,z_n)\in\mathbb C^n.
\]
Then $u$ is entire, $u(0)=1$, and $\bar\partial u=0$.  The function $u$ does not satisfy the above exponential growth condition,
since
$u(R,0,\ldots,0)=e^{e^R-1}.$

Moreover, for $R>\pi$, set
\[
x_R=\sqrt{R^2-\pi^2},
\qquad
z_R=(x_R+i\pi,0,\ldots,0),
\]
so that $|z_R|=R$. If $|z-z_R|<1$, write
$z_1=x_R+s+i(\pi+t)$ where $|s|<1$ and $|t|<1$. Then
\[
\operatorname{Re}(e^{z_1})
=
-e^{x_R+s}\cos t
\le
-e^{x_R-1}\cos 1.
\]
It follows that
\[
\sup_{|z-z_R|<1}|u(z)|
\le
e^{-e^{x_R-1}\cos 1-1}.
\]
Since $x_R\sim R$ as $R\to\infty$, the right-hand side decays faster
than $e^{-CR\ln R}$ for every $C>0$. Consequently,
 $\inf_{|z_0|=R}\sup_{|z-z_0|<1}|u(z)|$
does not admit a lower bound of the form $e^{-CR\ln R}$ with a fixed
constant $C$. 

A similar example with a nontrivial bounded potential is obtained by setting 
\[
u(z)=f(z)e^{g(z)}.
\] with $f=e^{e^{z_1}-1}$,  and  $g(z)=\sin(\operatorname{Re}z_1)$. 
Then $u\in C^\infty(\mathbb C^n)$, $u(0)=1$, and
\[
\bar\partial u
=
(\bar\partial g)u
=
Vu,
\]
where $V
=
\frac{1}{2}\cos(\operatorname{Re}z_1)\,d\bar z_1.$ 
In particular,
\[
\|V\|_{L^\infty(\mathbb C^n)}
\le \frac12.
\]
Since $-1\le g\le1$,
\[
e^{-1}|f(z)|
\le |u(z)|
\le e|f(z)|.
\]
Thus $u$ fails the exponential growth condition of
Theorem~\ref{hh}, just as $f$ does.  Moreover, 
\[ \sup_{|z-z_R|<1}|u(z)|
\le e \sup_{|z-z_R|<1}|f(z)| \le e^{-e^{x_R-1}\cos 1},
\]
which again decays faster than
$e^{-CR\ln R}$ for every $C>0$. 
\end{remark}

One may also consider   analogous bounds for the full gradient operator. As shown below, the full-gradient setting yields a substantially stronger lower bound than the $\bar\partial$ case.
\begin{lem}
     Suppose that $u=(u_1, \ldots, u_N) \in W_{loc}^{1,1}(\mathbb{R}^n)$ satisfies
\begin{equation}\label{gi}
    |\nabla u(x)|\le C|u(x)|\qquad \text{a.e. on}\ \ \mathbb R^n.
\end{equation}
Then $u$ admits a locally Lipschitz representative, still denoted by
$u$, such that either $u\equiv 0$, or else $u$ is nowhere vanishing and
\[
|u(0)|e^{-C|x|}
\le |u(x)|
\le |u(0)|e^{C|x|},
\qquad x\in\mathbb{R}^n.
\]
\end{lem}

\begin{proof}
By the Sobolev embedding theorem and \eqref{gi}, a standard bootstrap argument yields $u\in W_{ {loc}}^{1,\infty}(\mathbb R^n)$. Thus $u$ admits a locally Lipschitz representative. 
Set
\[
w(x)=|u(x)|.
\]
Then $w\in W_{ {loc}}^{1,\infty}(\mathbb R^n)$ and
\[
|\nabla w|
\le |\nabla u|
\le Cw
\qquad \text{a.e. in }\ \mathbb R^n.
\]
For $\varepsilon>0$, define
\[
v_\varepsilon(x)=\log\bigl(w(x)+\varepsilon\bigr).
\]
Then
\[
|\nabla v_\varepsilon|
=
\frac{|\nabla w|}{w+\varepsilon}
\le
 \frac{Cw}{w+\varepsilon}
\le C
\qquad \text{a.e. in }\ \mathbb R^n.
\]
Hence $v_\varepsilon$ satisfies
\[
|v_\varepsilon(x)-v_\varepsilon(y)|
\le C|x-y|,
\qquad x,y\in\mathbb R^n.
\]
Equivalently,
\begin{equation}\label{eps}
    e^{-C|x-y|}
\le
\frac{|u(x)|+\varepsilon}{|u(y)|+\varepsilon}
\le
e^{C|x-y|}.
\end{equation}

Suppose that $u(y)=0$ for some $y\in\mathbb R^n$. Then
\[
|u(x)|+\varepsilon
\le e^{C|x-y|}\varepsilon.
\]
Letting $\varepsilon\to0$ gives $u(x)=0$ for every
$x\in\mathbb R^n$. Thus either $u\equiv0$, or $u$ is nowhere
vanishing. In the latter case, letting $\varepsilon\to0$ in \eqref{eps} yields
\[
e^{-C|x-y|}
\le
\frac{|u(x)|}{|u(y)|}
\le
e^{C|x-y|},
\qquad x,y\in\mathbb R^n.
\]
Taking $y=0$, we obtain the desired inequality, which  completes the proof. 
 \end{proof}

As an application, Theorem \ref{hh} immediately yields the following unique continuation property for \eqref{eqn} at infinity. The $|z|\ln |z|$  decay assumption here is not optimal for this purpose:  as Theorem~\ref{ol} will show,  a sufficiently rapid linear exponential decay already forces triviality.

\begin{cor}\label{lan}
   Suppose $u\in L_{loc}^1(\mathbb C^n)$  satisfies \begin{equation*} 
        \bar\partial u = Vu \ \ \text{on}\ \ \mathbb C^n
    \end{equation*}  
in the sense of distributions for some   $(0,1)$-form $V\in L^\infty(\mathbb C^n)$.  Then there exists a  constant $\tilde C>0$   such that whenever  $$|u(z)|\le   e^{- \tilde C|z| |\ln|z|| },\quad   \ \ |z|\gg1,$$ then $u\equiv 0$. In particular, if   \begin{equation*}
    |u(z)|\le  e^{-|z|^{1+\epsilon}}, \ \ |z|\gg1 
\end{equation*}  for some $\epsilon>0$, then $u\equiv 0$. 
    \end{cor}
    
\begin{proof}    By  Lemma \ref{bs},   $u\in C(\mathbb C^n)$.    
We argue by contradiction that $u\not\equiv 0$. Choose $a\in\mathbb C^n$ such that $u(a)\ne0$
and define
\[
v(z):=\frac{u(z+a)}{u(a)},
\qquad
\tilde V(z):=V(z+a).
\]
Then
\[
v(0)=1,
\qquad
\|\tilde V\|_{L^\infty(\mathbb C^n)}
=
\|V\|_{L^\infty(\mathbb C^n)},
\]
and $$
\bar\partial v=\tilde Vv \qquad \text{on}\ \ \mathbb C^n$$
in the sense of distributions. 
Moreover, since $v$ tends to zero at infinity, it satisfies, for
example,
\[
|v(z)|\le e^{|z|}
\qquad\text{for }|z|\gg1.
\]
 Theorem \ref{hh} then provides a constant \(C>0\) dependent on $ \|V\|_{L^\infty(\mathbb C^n)}$ and $n$ such that, for all sufficiently large \(R\), and every  $z_0\in \mathbb C^n$ with $|z_0|=R$. 
   $$\sup_{|z-z_0|<1}|v(z)| \ge e^{-CR\ln R}. $$

On the other hand, for $|z_0|=R$ and $|z-z_0|<1$, $ |z+a| \ge R-|a|-1. $
Hence the assumed decay of $u$ gives
\[
\sup_{|z-z_0|<1}|v(z)|
\le
\frac{1}{|u(a)|}
e^{-\tilde C
(R-|a|-1)
\ln(R-|a|-1)}
\]
for all sufficiently large $R$. Since $(R-|a|-1)\ln(R-|a|-1)
\sim R\ln R $ as $R\to\infty,$ 
this contradicts the preceding lower bound provided
$\tilde C>C$. Therefore $u\equiv0$. Finally, the second assertion follows from the first since for every $\epsilon>0$ and every $\tilde C>0$,
 $|z|^{1+\epsilon}
\ge
\tilde C|z|\ln|z|$  
for all sufficiently large $|z|$. 
\end{proof}

\section{Proofs of Theorems  \ref{ol}  and   \ref{cr}}
Instead of proving Theorem \ref{ol} directly,  we establish unique continuation for a large class of radially decaying potentials described below, from which Theorem \ref{ol} follows immediately. 

\begin{theorem}\label{gen}
Let $\phi:[R_0,\infty)\to(0,\infty)$ be a nonincreasing
function for some $R_0\ge 0$, and set
\[
\Phi(r):=\int_{R_0}^r\phi(s) ds.
\]
Assume that
\begin{equation}\label{dc}
    \frac{\Phi(r)}{\ln r}\rightarrow\infty
\qquad\text{as }r\to\infty.
\end{equation}
Suppose $ u\in L^1_{ loc}
(\mathbb C^n\setminus\overline{B_R})$
satisfies
\[
\bar\partial u=Vu
\qquad\text{on }
\mathbb C^n\setminus\overline{B_R}
\]
in the sense of distributions, where
$V\in L^\infty
(\mathbb C^n\setminus\overline{B_R})$ 
is a $(0,1)$-form satisfying
\[
|V(z)|\le  {C_0}{\phi(|z|)},
\qquad |z|\gg1,
\]
for some $C_0>0$. If
\[
|u(z)|\le e^{-C \Phi(|z|)},
\qquad |z|\gg1,
\]
for some $C>2C_0,$ 
then $u$ vanishes identically.
\end{theorem}

 The condition \eqref{dc} is satisfied by many natural decay rates, as illustrated by the following examples, which yield corresponding unique continuation results.
\begin{example}\label{ex11}
   Suppose $ u\in L^1_{ loc}
(\mathbb C^n\setminus\overline{B_R})$ satisfies $\bar\partial u=Vu$ on $\mathbb C^n\setminus\overline{B_R}$ in the sense of distributions for some $(0,1)$-form $V\in L^\infty
(\mathbb C^n\setminus\overline{B_R})$. 
   \begin{enumerate}
  \item Let \[
\phi(r)\equiv1 \qquad\text{and}\qquad C_0 = \|V\|_{L^\infty
(\mathbb C^n\setminus\overline{B_R})}.
\]
Then
\[
\Phi(r)= \int_{R_0}^r\phi(s)\,ds = r+O(1),
\]
and the hypothesis on $V$ in Theorem \ref{gen} is automatically satisfied. Hence,     if $$|u(z)|\le e^{-C|z|},\qquad |z|\gg1  $$ for some $C>2\|V\|_{L^\infty
(\mathbb C^n\setminus\overline{B_R})}$, then $u\equiv 0$. Thus Theorem \ref{gen} recovers    Theorem \ref{ol}.

\item Let
\[
\phi(r)=\frac{1}{(\ln r)^\alpha},
\qquad \alpha>0.
\]
Then for   $R_0>1$,
\[
\Phi(r)=\int_{R_0}^r\phi(s)\,ds
=
\frac{r}{(\ln r)^\alpha}
\left(1+O\left(\frac{1}{\ln r}\right)\right).
\]
Consequently, Theorem \ref{gen} implies that if
\[
|V(z)|\le \frac{C_0}{(\ln|z|)^\alpha},
\qquad |z|\gg1,
\]
and
\[
|u(z)|
\le
e^{-\frac{C|z|}{(\ln|z|)^\alpha}
}, \qquad |z|\gg1,
\]
for some $C>2C_0$, then $u\equiv0$.  

\item Let
\[
\phi(r)=\frac{(\ln r)^\beta}{r},
\qquad \beta>0.
\]
 Then
\[
\Phi(r)=\int_{R_0}^r\phi(s)\,ds
=
\frac{(\ln r)^{\beta+1}}{\beta+1}+O(1).
\]
Choose $R_0$ sufficiently large so that $\phi$ is nonincreasing when $r\ge R_0$.  Hence, if
\[
|V(z)|
\le
C_0\frac{(\ln|z|)^\beta}{|z|},
\qquad |z|\gg1,
\]
and
\[
|u(z)|
\le
e^{-\frac{C}{\beta+1}
(\ln|z|)^{\beta+1}
},
\qquad |z|\gg1,
\]
for some $C>2C_0$, then $u\equiv0$.

\end{enumerate} 
 \end{example}

   To prove Theorem~\ref{gen}, we first establish the result in one complex dimension and then use slicing to reduce the general case to almost every complex line parallel to a coordinate axis. In one dimension, the similarity principle below shows that the zeros of a nontrivial solution are isolated and have positive multiplicities, which implies a monotonicity property for the winding number on concentric circles. Combined with the equation in polar coordinates, this yields a lower bound on the radial mean of $\ln|u|$ that is incompatible with the decay assumption in Theorem~\ref{gen}.

\begin{lem}\label{sp}
Let $\Omega\subset\mathbb C$ be a domain. Suppose
$ u\in L^1_{loc}(\Omega)$ satisfies
\[
\bar\partial  u=V u
\qquad\text{on }\ \ \Omega
\]
in the sense of distributions, where
$V\in L^\infty_{loc}(\Omega)$. Then the following hold.

\begin{enumerate}
    \item   For every relatively compact smooth domain $D\Subset\Omega$,
there exist $ \lambda\in W^{1,p}(D) $ for every $p<\infty, $ and a holomorphic function $h$ on $D$ such that
\[
u=e^\lambda h
\qquad\text{on }\ \ D.
\]
In particular, either $u\equiv0$, or the zeros of $u$ are isolated.
If $z_0$ is a zero of $u$, its multiplicity is defined by
\begin{equation}\label{ord}
    \operatorname{ord}_{z_0}u:=\operatorname{ord}_{z_0}h,
\end{equation}
and is a positive integer.

\item Let $0 < r_1< r_2$ be such that $\{z:r_1\le |z|\le r_2\}\Subset\Omega$ and $ u$ does not vanish on the circles $|z|=r_1$ and $|z|=r_2$. Denote by $k(r_j)$ the winding number of the curve
 $\theta\mapsto  u(r_je^{i\theta})$ 
about the origin. Then
\[
k(r_2)-k(r_1)
=
\sum_{r_1<|z|<r_2}\operatorname{ord}_z  u
\ge0.
\]

\end{enumerate}
\end{lem}

\begin{proof}
By  Lemma \ref{bs} part 2, $  u\in W^{1,p}_{  loc}(\Omega) $ for every $p<\infty$, and in particular $ u$ is continuous. Fix a
relatively compact smooth domain $D\Subset\Omega$. Let 
$$  \lambda(z) = -\frac{1}{\pi}
    \int_{D}
    \frac{V(\zeta)}{\zeta-z}\,dv_\zeta.$$
    Then $\lambda\in W^{1,p}(D)$  for every $p<\infty$,  and  solves  
$\bar\partial\lambda=V$ on $D$. See, for instance, \cite{Ve} and \cite{PZ4}. In particular,  $\lambda\in C(D)$.
By Lemma~\ref{pc},
\[
h:=e^{-\lambda} u
\]
is holomorphic on $D$. Since $e^\lambda$ is continuous and nowhere
vanishing, $ u$ and $h$ have the same zeros. Thus, unless $ u\equiv0$, its zeros are isolated.

For a zero $z_0$ of $u$,  the definition \eqref{ord} is independent of the factorization. Indeed, if
$u=e^{\lambda_1}h_1=e^{\lambda_2}h_2$ locally, then
$\bar\partial(\lambda_1-\lambda_2)=0$, so that
$\lambda_1-\lambda_2$ is holomorphic, and
\[
h_2=e^{\lambda_1-\lambda_2}h_1.
\]
Since the holomorphic factor $e^{\lambda_1-\lambda_2}$ is nowhere
vanishing, $h_1$ and $h_2$ have the same order at $z_0$.
Consequently, every zero of $u$ has positive integer multiplicity.

For the second assertion, choose a relatively compact smooth domain
$U\Subset\Omega$ containing $\{z:r_1\le |z|\le r_2\}. $ 
By the preceding construction, $ u=e^\lambda
h$ on $U$, with $h$ holomorphic.
Since $e^\lambda$ has winding number zero on every closed curve,
$ u$ and $h$ have the same winding number on the two boundary
circles. Hence the argument principle gives
\[
k(r_2)-k(r_1)
=
\sum_{r_1<|z|<r_2}\operatorname{ord}_z h
=
\sum_{r_1<|z|<r_2}\operatorname{ord}_z  u
\ge0.
\]
This proves the result.
\end{proof}

\begin{proof}[Proof of Theorem \ref{gen}: ]
We first consider the case $n=1$. By  Lemma \ref{bs} part 2, $u\in W^{1,p}_{ loc}
(\mathbb C\setminus\overline{D_R})$ 
for every $p<\infty$, and in particular $u$ is continuous.

Suppose that $u\not\equiv0$. By Lemma~\ref{sp}, the zeros of
$u$ are isolated. Choose $r_0>R$ sufficiently large so that
the assumed estimate for $V$ holds for $|z|>r_0$ and $u$ has
no zeros on $|z|=r_0$.
For $r>r_0$ such that $u$ has no zeros on $|z|=r$, let
$k(r)$ denote the winding number of $\theta\mapsto u(re^{i\theta})$ 
about the origin, and define
\[
m(r):=
\frac{1}{2\pi}
\int_0^{2\pi}
\ln|u(re^{i\theta})|\,d\theta.
\]
By the local factorization in the proof of  Lemma~\ref{sp}, $\ln | u| = \ln |h| +\operatorname{Re} \lambda    $, where $h\not\equiv 0$ is holomorphic. Since  $\ln |h|\in W^{1,1}_{ loc}$, it follows that   $\ln| u|\in W^{1,1}_{ loc}$, and hence $m$ is locally absolutely continuous. Moreover, for almost every $r>r_0$, the winding number 
\[
k(r)=
\frac{1}{2\pi i}
\int_0^{2\pi}
\frac{\partial_\theta  u(re^{i\theta})}
     { u(re^{i\theta})}\,d\theta.
\]

Since  $\bar\partial
=
\frac{e^{i\theta}}{2}
\left(
\partial_r+\frac{i}{r}\partial_\theta
\right)  $ in polar coordinates, for almost every $r>r_0$, the equation \eqref{eqn} is rewritten as 
\[
\frac{\partial_r u}{u}
+
\frac{i}{r}
\frac{\partial_\theta u}{u}
=
2e^{-i\theta}V.
\]
 Averaging over $\theta$ and taking real parts, we obtain
\[
m'(r)-\frac{k(r)}{r}
=
\frac{1}{2\pi}
\int_0^{2\pi}
2\operatorname{Re}
\left(
e^{-i\theta}V(re^{i\theta})
\right)\,d\theta.
\]
Therefore, by the monotonicity of $k$ in Lemma~\ref{sp},
\[
m'(r)
\ge
\frac{k(r)}{r}
- {2C_0}{\phi(r)} \ge
\frac{k(r_0)}{r}
- {2C_0}{\phi(r)}.
\]
Integrating from $r_0$ to $r$ yields
\[
m(r)
\ge
m(r_0)
+
k(r_0)\ln\frac{r}{r_0}
-
2C_0\int_{r_0}^r {\phi(s)}{ds}.
\]
Making use of the definition of $\Phi$ and its assumption, 
\[
m(r)\ge -2C_0 \Phi(r)+O(\ln r)\ge -(2C_0+o(1))\Phi(r)
\qquad\text{as }r\to\infty.
\]

On the other hand, the assumed decay of $u$ gives
\[
m(r)\le -C\Phi(r)
\]
for all sufficiently large $r$. Since $C>2C_0$, these two
estimates are incompatible as $r\to\infty$. Hence $ u\equiv0 $  on $
\mathbb C\setminus\overline{D_R}.$

We now assume  $n\ge 2$. Writing $V= \sum_{j=1}^n V_jd\bar z_j$, the equation \eqref{eqn} is equivalent to  
\begin{equation}\label{dp}
    \bar\partial_{z_j} u= V_j u, \qquad  j=1, \ldots, n
\end{equation}
in the sense of distributions on $\mathbb C^n\setminus\overline{B_R}$. We   consider slices of $\mathbb C^n$ parallel to the $z_1$-axis and  $z=(z_1,z'),$ with $ z'\in\mathbb C^{n-1}.$   For almost every fixed $z'\in\mathbb C^{n-1}$   with $|z'|<\frac{R}{2}$, Lemma~\ref{slice}, applied to the first equation of \eqref{dp}, gives $$\tilde u(\zeta):=u(\zeta,z')\in L^1_{ {loc}}(\mathbb C\setminus
\overline{D_{\sqrt{R^2-|z'|^2}}}), \qquad
     \tilde V(\zeta): =V_1(\zeta,z') \in L^\infty(\mathbb C\setminus
\overline{D_{\sqrt{R^2-|z'|^2}}})$$
and 
\[
    \bar\partial_\zeta \tilde u
    =
     \tilde V\tilde u\qquad \text{on}\ \ \mathbb C \setminus
\overline{D_{\sqrt{R^2-|z'|^2}}}
\]
in the sense of distributions.  
   
Let   $\rho=\sqrt{|\zeta|^2+|z'|^2}.$ Since $\rho\ge|\zeta|$  and $\phi$ is nonincreasing, for sufficiently large $|\zeta|$,
\[
| \tilde V(\zeta)|
\le
 {C_0}{\phi(\rho)}
\le
 {C_0}{\phi(|\zeta|)}.
\]
Moreover, since $\Phi$ is increasing,
\[
|\tilde u(\zeta)|
\le
e^{-C\Phi(\rho)}
\le
e^{-C\Phi(|\zeta|)}.
\]
 The one-dimensional result therefore applies and yields
 $\tilde u\equiv0$  
for almost every $z'\in\mathbb C^{n-1}$ with $|z'|<\frac{R}{2}$. Consequently, by Fubini's theorem, $u=0$ almost everywhere on the nonempty open set   $\{|z'|<\frac{R}{2}, \ |z_1|> \sqrt{R^2-|z'|^2}\}$. Hence $u\equiv 0$   in $\mathbb C^n\setminus\overline{B_R}$ by the weak unique continuation property in Theorem~\ref{pz}  part~1.
  \end{proof}

 \begin{remark}
 The condition \eqref{dc} is essential. Indeed, in the borderline case  
\[
\phi(r)=\frac1r,
\]
one has \[
\Phi(r)= \int_{1}^r\phi(s)\,ds
=\ln r 
\]
with $R_0=1$, so that the hypothesis \eqref{dc} fails. Moreover, for any $C>0$, choose an integer
$m\ge C$ and consider the nontrivial holomorphic function
\[
u(z)=z^{-m},\qquad |z|>1.
\]
Then  \[
|u(z)|=|z|^{-m}
\le |z|^{-C}
=e^{-C\Phi(|z|)},\qquad |z|>1.
\]
Thus, when \eqref{dc} is dropped, even an arbitrarily large constant
$C$ in the prescribed decay does not force vanishing. This shows that
the borderline scale $\phi(r)=\frac{1}{r}$ cannot in general be included.

 \end{remark}

\begin{remark}\label{gensharp}
The constant $2C_0$ in Theorem~\ref{gen} is sharp. Indeed, 
let
\[
u(z):=e^{-2C_0\Phi(|z|)},
\qquad |z|>R.
\]
Since $\Phi$ is locally absolutely continuous and
$\Phi'(r)= {\phi(r)}$, for almost every $r>R\ge R_0$, we have
$\bar\partial u=Vu$ 
almost everywhere on $|z|>R$, where
$V(z):
=
- C_0{\phi(|z|)}
\sum_{j=1}^n
\frac{z_j}{|z|}\,d\bar z_j.$ 
Thus $$|V(z)|= C_0{\phi(|z|)}.$$ 
Hence the strict threshold $C>2C_0$ cannot be improved.
\end{remark}

We next prove Theorem \ref{cr} for compactly supported potentials.   The proof follows the approach of \cite{CR}, while extending the result to higher dimensions and relaxing the assumption $V\in L^\infty(\mathbb C)$ imposed there to $V\in L^p(\mathbb C^n)$ for some $p>2$.

 \begin{proof}[Proof of Theorem \ref{cr}: ]We first prove the Theorem when $n=1$. 
 By Lemma \ref{bs} part 1,  $ u\in W_{loc}^{1, p}(\mathbb C)$.  In particular,   $u$ is  continuous on $\mathbb C$. Since $\lim_{z\rightarrow \infty} u =0,$ it follows that $u$ is bounded on $\mathbb C$.

Define
\begin{equation} \label{lbc}
\lambda(z)
=
-\frac{1}{\pi}
\int_{\mathbb C}
\frac{V(\zeta)}{\zeta-z}\,dv_\zeta,
\qquad z\in\mathbb C.
\end{equation}
Since $V\in L^p(\mathbb C)$ for some $p>2$ and has compact support,
we have  $\lambda\in W_{ {loc}}^{1,p}(\mathbb C)$  
and
\[
\bar\partial\lambda=V\qquad \text{in}\ \ \mathbb C
\]
in the sense of distributions. In particular, by the Sobolev
embedding theorem, $\lambda$ is continuous on $\mathbb C$.
Moreover, the compact support of $V$  implies from \eqref{lbc} that
$\lim_{|z|\to\infty}\lambda(z)=0.$ 
Hence $\lambda$ is bounded on $\mathbb C$.

By  Lemma \ref{pc}, \[
\bar\partial(ue^{-\lambda})=0
\qquad\text{in }\mathbb C
\]
in the sense of distributions.
Thus    $ ue^{-\lambda} $ is holomorphic on $\mathbb C$. On the other hand, since   $\lambda$ is bounded and $\lim_{|z|\to\infty}u(z)=0,$ we have  \[\lim_{z\rightarrow \infty}  u(z)e^{-\lambda(z)} =0. \]
Liouville's theorem therefore yields
 $ ue^{-\lambda}\equiv0.$
 Since $e^{-\lambda}$ never vanishes, it follows that
 $u\equiv0$ on $\mathbb C$.

The  case  $n\ge 2$ follows from the one-dimensional result by the same coordinate slicing argument as in the proof  of Theorem \ref{gen}. Write $V=\sum_{j=1}^n V_j\,d\bar z_j$  and decompose $z=(z_1,z')$, where $z'\in\mathbb C^{n-1}$.
For almost every fixed $z'\in\mathbb C^{n-1}$,  Lemma~\ref{slice} gives
\[
\tilde  u(\zeta)=u(\zeta,z')\in L_{ {loc}}^2(\mathbb C),
\qquad
\tilde  V (\zeta)=V_1(\zeta,z')\in L^p(\mathbb C),
\]
and \[
\bar\partial_\zeta\tilde  u
=
\tilde  V\,\tilde  u
\qquad\text{in }\mathbb C
\]
in the sense of distributions. 
  Moreover, since $V$ is compactly supported
in $\mathbb C^n$, the function $\tilde  V $ is compactly supported
in $\mathbb C$ for almost every $z'$. For each fixed $z'$,  as $|\zeta|\to\infty$, $|(\zeta,z')|\rightarrow\infty$, and thus by assumption 
\[
\lim_{\zeta\rightarrow \infty}\tilde  u(\zeta) =0.
\]
The one-dimensional result therefore yields  $\tilde  u\equiv0$  
for almost every $z'\in\mathbb C^{n-1}$. Hence, by Fubini's theorem,
$u\equiv 0$   in $\mathbb C^n$.  
  \end{proof}

\section{Proof of Theorem   \ref{main4n} }
In this section, we   prove Theorem \ref{main4n} with $L^2$ potentials  under $L^2$-flatness decay assumption. We first establish    a one-dimensional unique continuation result at infinity. This is obtained   by inversion from   Theorem \ref{pz} part 2 at a finite point. To pass  to higher dimensions using the coordinate slicing,    we   further show, using Fubini's theorem and a dyadic argument,  that the global $L^2$-flatness of $u$ implies $L^2$-flatness on almost every such slice. The one-dimensional result can then be applied on almost every slice.

\begin{proof}[Proof of Theorem \ref{main4n}: ]
We first prove the case $n=1$.  Let $v(z): = u(\frac{1}{z}) $ and $ W(z): = -\frac{V(\frac{1}{z})}{\bar z^2}$ on $D_\frac{1}{R}\setminus\{0\}$.  Then   by Lemma \ref{rev}, $ v\in W_{loc}^{1,2}(D_\frac{1}{R}\setminus\{0\})   $ and  $$ \bar\partial v     = W  v  \ \ \text{on}\ \  D_\frac{1}{R}\setminus\{0\}  $$
 in the sense of distributions. 
We next verify that $v$ satisfies the hypothesis in Theorem \ref{pz} part 2.

Since $u \in L^2(\mathbb C\setminus \overline{D_R})$,   making use of change of variables, we infer  
$$\int_{D_\frac{1}{R}}|v(z)|^2dv_z = \int_{\mathbb C\setminus \overline{D_R}}\frac{|v(\frac{1}{z})|^2}{|z|^4}dv_z = \int_{\mathbb C\setminus \overline{D_R}}\frac{|u(z)|^2}{|z|^4}dv_z\le \max\{1, R^{-4}\}\int_{\mathbb C\setminus \overline{D_R}}|u(z)|^2dv_z <\infty. $$ 
Similarly, since $V\in L^2(\mathbb C \setminus \overline{D_R})$
and 
$$\int_{D_\frac{1}{R} } |W(z)|^2dv_z = \int_{D_\frac{1}{R} }   \frac{|V(\frac{1}{z})|^2 }{|z|^4}  dv_z = \int_{\mathbb C\setminus \overline{D_R} } |V(z)|^2 dv_z <\infty.$$
In particular, $v \in L^2(D_\frac{1}{R})$ and $Wv\in L^1(D_\frac{1}{R})$. 

We then apply a removable singularity result of Harvey-Polking \cite{HP} (see also \cite[Lemma A.1]{PZ}) to obtain  \begin{equation}\label{r1}
    \bar\partial v =  Wv \ \ \text{on}\ \ D_\frac{1}{R} 
\end{equation}   
in the sense of distributions. Moreover, since 
$$ \int_{D_\frac{1}{R}\ }|\nabla v(z)|^2dv_z = \int_{\mathbb C\setminus \overline{D_R}}\frac{|\nabla v(\frac{1}{z})|^2}{|z|^4}dv_z = \int_{\mathbb C\setminus \overline{D_R}}|\nabla u(z)|^2dv_z<\infty, $$ it follows that 
$$v\in W^{1,2}(D_\frac{1}{R}).$$
On the other hand,  the $L^2$ flatness assumption of $u$ at infinity further leads to  
$$ r^{-m}\int_{|z|<r}|v(z)|^2 dv_z = r^{-m}\int_{|z|>\frac{1}{r}}\frac{|u(z)|^2}{|z|^4}dv_z \le r^{4-m}\int_{|z|>\frac{1}{r}}|u(z)|^2dv_z \rightarrow 0$$
as $r\rightarrow 0$.

Altogether we have $v\in W^{1,2}(D_{\frac{1}{R}})$ satisfying \eqref{r1} for some $W\in L^2(D_{\frac{1}{R}}) $, and $v$ vanishes to infinite order in the $L^2$ sense  at $0$. Making use of Theorem \ref{pz} part 2  we have $v\equiv 0$ on $ D_{\frac{1}{R}}$. Thus $u = 0$ on $ \mathbb C\setminus \overline{D_R}$.

 We next assume $n\ge2$.
Write $ V=\sum_{j=1}^n V_j\,d\bar z_j$
and decompose $z=(z_1,z')$, where $z'\in\mathbb C^{n-1}$.
For almost every fixed $z'$ with $|z'|<R/2$, set
\(\tilde u(\zeta):=u(\zeta,z'),
\ 
\tilde V(\zeta):=V_1(\zeta,z').
\) 
By Lemma \ref{slice},   
\[
\tilde u
\in
W^{1,2}\left(
\mathbb C\setminus
\overline{D_{\sqrt{R^2-|z'|^2}}}
\right), \qquad 
\tilde V
\in
L^2\left(
\mathbb C\setminus
\overline{D_{\sqrt{R^2-|z'|^2}}}
\right)
\]
for almost every such $z'$ and satisfies
\[
\bar\partial_\zeta\tilde u
=
\tilde V\,\tilde u\qquad \text{on}\ \ \mathbb C\setminus
\overline{D_{\sqrt{R^2-|z'|^2}}}
\]
 in the sense of distributions.

We claim that, for almost every such $z'$, $\tilde u$ vanishes
to infinite order at infinity in the $L^2$ sense. Let
\[
F(r):=\int_{|z|>r}|u(z)|^2\,dv_z
\]
and
\[
F_{z'}(r):=
\int_{|\zeta|>r}
|\tilde u(\zeta)|^2\,dv_\zeta.
\]
For every $r>R$,
\[
\int_{|z'|<R/2}F_{z'}(r)\,dv_{z'}
\le F(r).
\]
Fix a positive integer $m$. By \eqref{flatn}, for all sufficiently
large integers $j$,
\[
F(2^j)\le 2^{-j(m+2)}.
\]
Hence
\[
\sum_j
2^{jm}
\int_{|z'|<R/2}F_{z'}(2^j)\,dv_{z'}\le \sum_j 2^{jm} F(2^j)
<\infty.
\]
By Tonelli's theorem, for almost every $z'$ with $|z'|<R/2$,
\[
2^{jm}F_{z'}(2^j)\rightarrow0
\qquad\text{as }j\to\infty.
\]
Since $F_{z'}(r)$ is nonincreasing in $r$, it follows that
\[
r^mF_{z'}(r)\rightarrow0
\qquad\text{as }r\to\infty.
\]
Taking the intersection over $m\in\mathbb N$, we conclude that
$\tilde u$ is $L^2$-flat at infinity for almost every $z'$.

The one-dimensional result therefore gives
$\tilde u\equiv0$ 
for almost every $z'$ with $|z'|<R/2$. By Fubini's theorem, $u$
vanishes almost everywhere on a nonempty open subset of
$\mathbb C^n\setminus\overline{B_R}$. The weak unique continuation
property in Theorem~\ref{pz}  part~1, then yields
 $u\equiv0$ on $\mathbb C^n\setminus\overline{B_R}.$
\end{proof}

The $L^2$-flatness condition in Theorem~\ref{main4n} is satisfied
under a variety of natural pointwise decay assumptions, as illustrated
by the following examples.

\begin{example}
Suppose
$
u=(u_1,\ldots,u_N)^T
\in W^{1,2}(\mathbb C^n\setminus\overline{B_R})
$
satisfies
$
\bar\partial u=Vu$ on $\mathbb C^n\setminus\overline{B_R},
$
where
$
V\in L^2(\mathbb C^n\setminus\overline{B_R})
$
is a matrix-valued $(0,1)$-form. Suppose that $u$ decays faster than every algebraic rate, namely, for every $k>0$ there exists $C_k>0$ such that
\[
|u(z)|\le C_k|z|^{-k},
\qquad |z|\gg1.
\]
Given $m>0$, choose
\[
k>\frac{m+2n}{2}.
\]
Then
\[
r^m\int_{|z|>r}|u(z)|^2\,dv_z
\lesssim
r^m\int_r^\infty s^{2n-1-2k}\,ds
\lesssim
r^{m+2n-2k}
 \rightarrow0.
\]
Thus \eqref{flatn} holds, and Theorem~\ref{main4n} yields $u\equiv0$.

In particular, this applies to solutions satisfying
\[
|u(z)|\le e^{-c|z|^\alpha},
\qquad |z|\gg1,
\]
for some $c,\alpha>0$, as well as
\[
|u(z)|\le e^{-c(\ln|z|)^{1+\beta}},
\qquad |z|\gg1,
\]
for some $c,\beta>0$, since both decay faster than every negative power of $|z|$.
\end{example}

  The following example shows that the conclusion of Theorem~\ref{main4n} may fail for potentials outside $L^2$, even when the potential belongs to $L^p$
 for some $p>2n$.

\begin{example}\label{ex3}
For each $p>2n$, choose
$
\epsilon\in\left(0,\frac{p-2n}{p}\right),
$
so that $(1-\epsilon)p>2n$. Define
\[
u_\epsilon(z):=e^{-|z|^\epsilon},
\qquad z\in\mathbb C^n\setminus\overline{B_1}.
\]
Then
$
u_\epsilon\in
W^{1,2}(\mathbb C^n\setminus\overline{B_1})
$
and vanishes to infinite order at infinity in the $L^2$ sense.
Moreover,
$
\bar\partial u_\epsilon=Vu_\epsilon,
$
where
\[
V=
-\frac{\epsilon}{2}|z|^{\epsilon-2}
\sum_{j=1}^n z_j\,d\bar z_j
\]
satisfies
$
|V(z)|
=
\frac{\epsilon}{2|z|^{1-\epsilon}}
$
and
$
V\in L^p(\mathbb C^n\setminus\overline{B_1}).
$
Thus for every $p>2n$, there exist potentials $V\in L^p(\mathbb C^n\setminus \overline{B_1})\setminus L^2(\mathbb C^n  \setminus\overline{B_1} )$ for which the conclusion of Theorem~\ref{main4n} fails.
\end{example}

The following example shows that unique continuation may fail for solutions with decay of finite algebraic order, and hence illustrates the essential role of the $L^2$-flatness assumption in Theorem~\ref{main4n}.

\begin{example}
For each $\alpha \in (0,\frac{1}{2})$,  define
$$u =\frac{1}{z e^{(\ln|z|^2)^\alpha}}.$$
Then $u\in W^{1,2}(\mathbb C\setminus \overline{D_2})$  and satisfies $\bar\partial u = V u$ on $\mathbb C\setminus \overline{D_2}$, with 
$$ V= \frac{-\alpha}{\bar z(\ln|z|^2)^{1-\alpha}}\in L^2(\mathbb C\setminus \overline{ D_2}).$$
Moreover,   
$u$ does not vanish to infinite order at infinity in the $L^2$ sense. 
\end{example}

\begin{proof}
The function $u$ is smooth on $\mathbb C\setminus \overline{D_2}  $.  A direct computation gives  \begin{equation*}
    \begin{split}
        &\bar\partial u =\frac{-\alpha}{ |z|^2(\ln|z|^2)^{1-\alpha}e^{(\ln|z|^2)^\alpha}};\\
       & \partial u =  -\frac{1}{z^2e^{(\ln|z|^2)^\alpha}}- \frac{\alpha}{ z^2(\ln|z|^2)^{1-\alpha}e^{(\ln|z|^2)^\alpha}}.
    \end{split}
\end{equation*} Hence  $u$ satisfies $ \bar\partial u = V u $ on $\mathbb C\setminus \overline{D_2}$. 

We next verify the required integrability.  
Since for every $ k\in \mathbb Z^+$, $e^{ 2(2\ln s)^\alpha}> (\ln s)^{k\alpha}$   for all sufficiently large \(s\), we may choose  $k>\frac{1}{\alpha}$ and  obtain
\begin{equation*}
    \int_{|z|>2}|u|^2dv_z \lesssim \int_{2}^\infty  \frac{ds}{ se^{ 2(2\ln s)^\alpha}}\lesssim \int_{2}^\infty  \frac{ds}{ s(\ln s)^{k\alpha}}  \lesssim   \int_{\ln 2}^\infty s^{-k\alpha}  ds \lesssim (\ln 2)^{1- k\alpha }<\infty.
\end{equation*}
On the other hand, using the facts that $e^{ 2(2\ln s)^\alpha}>1$, $\ln s>  1$ for all sufficiently large \(s\) and $0<\alpha<\frac{1}{2}$, 
$$  \int_{|z|>2}|\nabla u|^2dv_z \lesssim \int_{|z|>2} \frac{1}{|z|^{4} } dv_z\lesssim    \int_{2}^\infty s^{-3}  ds<\infty. $$
Similarly, 
\begin{equation*}
    \begin{split}
        \int_{\mathbb C\setminus \overline{D_2}} |V|^2dv_z \lesssim\int_2^\infty \frac{1}{s (\ln s)^{2-2\alpha}}ds = \int_{\ln 2}^\infty  s^{2\alpha-2}ds \lesssim  (\ln 2)^{2\alpha-1}<\infty.
    \end{split}
\end{equation*}
Together, we have $u\in W^{1,2}(\mathbb C\setminus \overline{D_2})$ and $V\in L^2(\mathbb C\setminus \overline{D_2}).$

It remains to show that \(u\) does not vanish to infinite order at infinity in the $L^2$ sense. When $r $ is sufficiently small, we have
$$  \int_{|z|>\frac{1}{r}}|u|^2 dv_z \gtrsim \int_{\frac{1}{r}}^\infty  \frac{1}{ se^{ 2(2\ln s)^\alpha}}ds.$$
Since \(0<\alpha<\frac{1}{2}\), for all sufficiently large \(s\),   $
e^{ 2(2\ln s)^\alpha}< s $. Therefore, for \(r>0\) sufficiently small,
$$  
\int_{|z|>\frac{1}{r}}|u|^2 dv_z \gtrsim \int_{\frac{1}{r}}^\infty  \frac{1}{ s^{2}}ds\gtrsim  r .$$
In particular,  taking any \(m>1\), we get
\[
r^{-m}\int_{|z|>\frac{1}{r}}|u|^2\,dv_z
\gtrsim r^{1-m}\to\infty
\quad \text{as } r\to 0.
\]
The proof is complete. 
\end{proof}

Despite the   failure of the unique continuation property in Theorem \ref{main4n}   when $V\notin L^2$ near infinity as demonstrated above, unique
continuation may still hold under additional structural assumption on the potential, for instance  when  the potential $V=O(\frac{1}{|z|})$ at infinity. Such condition allows potentials outside $L^2 $; indeed, the model case $\frac{1}{|z|}$  belong to   $L^p  \setminus L^2 $   for all $2n<p<\infty$ on    $  \mathbb C^n\setminus \overline{B_R}$.

\begin{theorem}\label{main3}
  Suppose $u=(u_1, \ldots, u_N)^T \in W^{1,2}(\mathbb C^n\setminus \overline{B_R})$   satisfies \begin{equation*} 
        \bar\partial u = Vu \ \ \text{on}\ \  \mathbb C^n \setminus \overline{B_R}
    \end{equation*}  for some  measurable matrix-valued  $(0,1)$-form $V$ on $\mathbb C^n\setminus \overline{B_R}$   with $|V| \le  \frac{C}{|z|}$   for some $C>0$. Assume that  $u$ vanishes to infinite order at infinity in the $L^2$ sense. Then  $u$ vanishes identically if either $N=1$, or if  $N\ge 2$ and   $C<\frac{1}{4}$. 
\end{theorem}

\begin{proof}
When  $n=1$,  define $w(z): = u(\frac{1}{z}) $ and $ W(z): = -\frac{V(\frac{1}{z})}{\bar z^2}   $ on $D_\frac{1}{R}\setminus\{0\}$.  As in the proof of Theorem~\ref{main4n},  $w\in W^{1,2}(D_{\frac{1}{R}})$ and   vanishes to infinite order at $0$. Moreover,  
$\bar\partial w=Ww$ with 
$|W(z)| = \frac{|V(\frac{1}{z})|}{|z|^2}\le \frac{C}{|z|}$ on $D_{\frac{1}{R}} $. Making use of Theorem \ref{pz} part 3, we have $w\equiv 0$ on $ D_{\frac{1}{R}}$. Thus $u = 0$ on $ \mathbb C\setminus \overline{D_R}$.   

The case $n\ge2$ is reduced to the one-dimensional result by the same slicing argument as in the proof of Theorem~\ref{main4n}. Write  $V=\sum_{j=1}^n V_j\,d\bar z_j$ and decompose $z=(z_1,z')$, where $z'\in\mathbb C^{n-1}$. For almost every fixed $z'$ with $|z'|<R/2$, the slice \(  \tilde u(\zeta):=u(\zeta,z')
\)
belongs to
 $W^{1,2}\left(
\mathbb C\setminus
\overline{D_{\sqrt{R^2-|z'|^2}}}
\right)$ 
and satisfies
\[
\bar\partial_\zeta   \tilde u
=
  \tilde V\,  \tilde u,
\qquad
  \tilde V(\zeta):=V_1(\zeta,z'),
\]
with
\[
|  \tilde V(\zeta)|
\le
\frac{C}{\sqrt{|\zeta|^2+|z'|^2}}
\le
\frac{C}{|\zeta|}.
\]
As in the proof of Theorem~\ref{main4n}, $  \tilde u$ is $L^2$-flat at infinity for almost every such $z'$. The one-dimensional result therefore gives $  \tilde u\equiv0$ for almost every $z'$ with $|z'|<R/2$. By Fubini's theorem, $u$ vanishes on a nonempty open subset of $\mathbb C^n\setminus\overline{B_R}$, and Theorem~\ref{pz}, part~1, then yields \(u\equiv0\)  on $\mathbb C^n\setminus\overline{B_R}.$
\end{proof}

As an immediate consequence of Theorem~\ref{main4n}, we obtain the following unique continuation result at infinity for the full gradient operator.

\begin{cor}
Suppose $u=(u_1, \ldots, u_N)^T \in W^{1,2}(\mathbb R^{2n}\setminus \overline{B_R} )$   satisfies $ |\nabla u|\le V|u|$
 for some nonnegative $V\in L^2(\mathbb R^{2n}\setminus \overline{B_R})$.  If $u$ vanishes to infinite order at infinity in the $L^2$ sense, then $u$ vanishes identically.
\end{cor}

\begin{proof}
Identify $\mathbb R^{2n}$ with $\mathbb C^n$. Then 
$$  |\bar\partial u| \le |\nabla u|\le V|u|.$$
The conclusion now follows from
Theorem~\ref{main4n}.
\end{proof}

Theorem \ref{main4n}  can also   be  applied to study the uniqueness of    certain nonlinear PDEs concerning Laplacian  under  sufficiently rapid decay at infinity. 

\begin{cor}\label{non}
Let $\psi =(\psi_1, \ldots, \psi_N)^T\in L_{loc}^1((\mathbb R^{2}\setminus \overline{D_R})\times \mathbb R^{2N})$ satisfy       \begin{equation*} 
    |\psi(x, y_1)-\psi(x, y_2)|\le V(x)|y_1-y_2|, \ \ \text{for a.e.}\ \ x\in  \mathbb R^{2}\setminus \overline{D_R},\  y_1, y_2 \in \mathbb R^{2N}
    \end{equation*} for some nonnegative  $V\in L^2(\mathbb R^{2}\setminus \overline{D_R})$. Then there exists at most one $  W^{2,2}(\mathbb R^2\setminus \overline{D_R} )$ real-valued vector solution $u=(u_1, \ldots, u_N)^T$    to  \begin{equation*}
     \Delta  u = \psi(\cdot, \nabla u)\ \ \text{on}\ \ \mathbb R^2\setminus \overline{D_R} 
 \end{equation*}  
such that $\nabla u$ vanishes to infinite order at infinity in the $L^2$ sense.
\end{cor}

\begin{proof}  Suppose  that   $u_1$ and $u_2$ both satisfy $ \Delta  u = \psi(x, \nabla u)$ on $\mathbb R^2  \setminus \overline{D_R}$, and $   \nabla u_j(x) $ vanishes to infinite order at infinity in the $L^2$ sense, $j=1, 2$. Let  $v: = u_1-u_2$. Then   
$$|\Delta v(x)| = | \psi(x, \nabla u_1(x)) -  \psi(x, \nabla u_2(x))|\le V(x)|\nabla v(x)|\ \ \text{on}\ \ \mathbb R^2\setminus \overline{D_R}, $$
 and $\nabla v$  is
$L^2$-flat at infinity. 

    With the convention
 $\bar\partial=\frac12(\partial_{x_1}+i\partial_{x_2}),
$  we have  $\Delta v(x) = 4\bar\partial \partial v(z) $. Since $v$ is real-valued componentwise,  $|\nabla v(x)| = 2|\partial v(z)|. $ Let  $$ \tilde v: = \partial v.$$ 
Then  $\tilde v\in W^{1, 2}(\mathbb R^2\setminus \overline{D_R})$, and    the preceding inequality   implies 
    $$ |\bar\partial \tilde v|\le \frac{1}{2}V |\tilde v|\ \ \text{on}\ \ \mathbb R^2\setminus \overline{D_R}  $$
    Moreover, since $|\tilde v|=\frac12|\nabla v|$,  $\tilde v$  is $L^2$-flat at infinity. 
    Applying Theorem \ref{main4n}, we obtain $\tilde v \equiv 0$. Since \(v\) is real-valued and $v\in L^2(\mathbb R^2\setminus \overline{D_R})$, this implies that \(v\equiv 0\).  
\end{proof}

\bibliographystyle{alphaspecial}

\fontsize{11}{11}\selectfont

\vspace{0.7cm}
\noindent pan1@pfw.edu,

\vspace{0.2 cm}

\noindent Department of Mathematical Sciences, Purdue University Fort Wayne, Fort Wayne, IN 46805-1499, USA.\\

\noindent zhan1313@pfw.edu,

\vspace{0.2 cm}

\noindent Department of Mathematical Sciences, Purdue University Fort Wayne, Fort Wayne, IN 46805-1499, USA.\\

\end{document}